\documentclass [12pt]{article}
\usepackage {a4}
\usepackage {amsfonts}
\usepackage {amssymb,amsmath,amsthm}
\usepackage [utf8]{inputenc}
\usepackage [english]{babel}
\usepackage{enumerate}
\usepackage{graphicx}
\usepackage{fullpage}
\usepackage{color}
\usepackage{verbatim}
\usepackage{url}
\usepackage{hyperref}

\makeatletter
\def\blfootnote{\gdef\@thefnmark{}\@footnotetext}
\makeatother

\author{Marcin Kotowski \footnote{{\tt mk249015@students.mimuw.edu.pl}} \quad \, Micha\l{} Kotowski \footnote{{\tt mk249019@students.mimuw.edu.pl}} \\ [+0.8ex]
\small Faculty of Mathematics, Informatics and Mechanics \\
\small University of Warsaw, \\
\small Banacha 2, 02-097 Warszawa, Poland.
}

\title{Random groups and property (T): \.Zuk's theorem revisited}

\newtheorem{theorem}{Theorem}[section]
\newtheorem{theoremalph}{Theorem}[section]

\newtheorem{remark}[theorem]{Remark}
\newtheorem{lemma}[theorem]{Lemma}
\newtheorem{corollary}[theorem]{Corollary}
\theoremstyle{definition}
\newtheorem{definition}[theorem]{Definition}

\newcommand{\pres}[2]{\langle #1 \vert #2 \rangle}

\newcommand{\gromov}[3]{\mathcal{G}(#1,#2,#3)}
\newcommand{\triangular}[2]{\mathcal{M}(#1,#2)}
\newcommand{\triangularplus}[2]{\mathcal{M}^{+}(#1,#2)}
\newcommand{\permutation}[2]{\mathcal{F}(#1,#2)}
\newcommand{\permutationr}[2]{\mathcal{F}^{red}(#1,#2)}
\newcommand{\permgraph}[2]{\mathcal{L}(#1,#2)}
\newcommand{\permrgraph}[2]{\mathcal{L}^{red}(#1,#2)}
\newcommand{\tripartite}[2]{G_3(#1, #2)}
\newcommand{\tripartiter}[2]{G_3^{red}(#1, #2)}
\newcommand{\hypergraph}[1]{L_3(#1)}
\newcommand{\scalar}[2]{\langle #1, #2 \rangle}

\newcommand{\Ss}{\mathcal{S}}
\newcommand{\R}{\mathcal{R}}
\newcommand{\Ll}{\mathcal{L}}

\newcommand{\norm}[1]{\Vert #1 \Vert}

\begin{document}

\blfootnote{2010 Mathematics Subject Classification 20F65.}

\maketitle

\begin{abstract}
We provide a full and rigorous proof of a theorem attributed to \.Zuk, stating that random groups in the Gromov density model for $d > 1/3$ have property (T) with high probability. The original paper had numerous gaps, in particular, crucial steps involving passing between different models of random groups were not described. We fix the gaps using combinatorial arguments and a recent result concerning perfect matchings in random hypergraphs. We also provide an alternative proof, avoiding combinatorial difficulties and relying solely on spectral properties of random graphs in $G(n, p)$ model. 
\end{abstract}

\section{Introduction}\label{ch1}

The topic of this paper are two important concepts in geometric group theory --- random groups and property (T).

Random groups have been an active field of study ever since their introduction by Gromov in the early nineties \cite{gromov}. The main focus of this theory is on asymptotic properties of groups. Roughly speaking, we define a model in which group presentations are chosen at random, subject to certain restrictions, and depend on some parameter going to infinity. We then study the probability that a random group has given property (for example, is trivial, free, word hyperbolic etc.). It turns out that frequently this probability tends to $0$ or $1$, so we can say that asymptotically ``almost none'' or ``almost all'' groups have the considered property.

Various different models of choosing a group at random have been proposed. The most popular and widely studied is the {\it Gromov density model} \cite{gromov}, in which the presentations have fixed number of generators, but the length and number of relators go to infinity. The number of relators is controlled by a parameter $d \in [0, 1]$, called the density. It has been discovered that many properties exhibit a ``phase transition'' --- below a certain critical density almost all groups have the property, while above it almost none do. The most celebrated example is Gromov's theorem (\cite{gromov}, see \cite{oll-update} for $d=\frac{1}{2}$), which states that for $d \geq \frac{1}{2}$ a random group in this model is almost surely trivial or $\mathbb{Z}_{2}$, while for $d < \frac{1}{2}$ it is almost surely infinite and word hyperbolic (where ``almost surely'' means ``with probability tending to $1$ as the relevant parameters go to infinity''). Therefore the Gromov model gives us new examples of hyperbolic groups.

An excellent survey of random groups is \cite{ollivier2005}, where the reader can find out more about the history of this field, important results and further references.

The second concept, {\it property (T)}, was first introduced by Kazhdan to study certain properties of lattices in Lie groups (hence it is also known as {\it Kazhdan's property}). Later it turned out that this property, concerning unitary representations of groups and their invariant vectors, is relevant to many other areas of mathematics and also computer science. A comprehensive treatment of this topic and its history can be found in \cite{bekka2008kazhdan}, here we will only mention that property (T) has found applications in representation theory, harmonic analysis, ergodic theory, expanding graphs, computational group theory and measure theory. For example, groups with property (T) were used by Margulis to give a first explicit construction of families of expanders, graphs with good spectral properties which are important in theoretical computer science (we will also encounter spectral graph theory in later sections). Many such applications are described in \cite{lubotzky}.

Interestingly, for a long time few examples of groups having property (T) were known, most of them associated with lattices in Lie groups. As we will see, it is random group theory that gives us new examples of such groups. Actually, in a suitably defined sense ``most'' groups have property (T), in the same way as ``most'' groups in the Gromov model are trivial or hyperbolic for certain densities.

This remarkable result is due to \.Zuk and was proved in \cite{zuk}. In this paper \.Zuk shows this for the {\it triangular model}, a model of random groups closely related to the Gromov model:

\begin{theoremalph}\label{th:theoremA}
For density $d > \frac{1}{3}$, a random group in the triangular model $\triangular{m}{d}$ has property \emph{(T)} with overwhelming probability.
\end{theoremalph}

Analogous statement for the Gromov model is attributed to \.Zuk, although it does not appear in \cite{zuk}:

\begin{theoremalph}\label{th:theoremB}
For density $d > \frac{1}{3}$, a random group in the Gromov model $\gromov{n}{l}{d}$ has property \emph{(T)} with overwhelming probability.
\end{theoremalph}

Together with Gromov's theorem, this shows that for $d$ between $\frac{1}{3}$ and $\frac{1}{2}$ we obtain a multitude of infinite hyperbolic groups with property (T). On the other hand, it is known \cite{oll-wise} that for $d < \frac{1}{5}$ random groups in the density model do not have property (T). It is an open problem to decide if the density $\frac{1}{3}$ for the Gromov model can be improved or if it is critical for having property (T).

\.Zuk's key insight is the introduction of the triangular model, mentioned above, and the use of a certain spectral criterion for property (T), which relates this property to the spectrum of the discrete Laplacian on a suitably defined finite graph. It is then possible to pass to the original Gromov model, so that the task of proving property (T) is reduced to the analysis of eigenvalues of the Laplacian on random graphs.

Unfortunately, \.Zuk's original paper contains gaps --- the steps needed to pass from results about random graphs to the triangular model are not spelled out explicitly and are stated without proof. The passage from the triangular model to the Gromov density model is not mentioned there, although it can be done following ideas from the survey \cite{ollivier2005}. As Ollivier points out in his survey, each of these steps requires dealing with technical details, but the required proofs do not seem to appear anywhere in the literature. In particular the passage from the results about random regular graphs to the triangular model turns out to be nontrivial.

The goal of this paper is to fill the gaps from \.Zuk's argument and describe the passage to the Gromov model, thus giving a full and rigorous proof of Theorem \ref{th:theoremB}.  We introduce the appropriate definitions and properties of random group models and then proceed to prove in detail the facts needed to obtain the main theorem, relying on a strong result from random graph theory to deal with the major difficulty mentioned above. In addition we give an alternative way of passing to the triangular model, not relying on theorems about random regular graphs, but using more recent results about the so called Erd\H{o}s-Renyi model of random graphs.

The structure of the paper is as follows. In Section \ref{ch2} we introduce the models of random groups, property (T), the spectral criterion and give an outline of \.Zuk's approach to the proof. In Section \ref{ch3}, after recalling necessary notions from graph theory, we prove that random groups in the triangular model typically have property(T). We then use this result to finish the proof of the main theorem, proving that groups in the Gromov model are quotients of groups in the triangular model. Section \ref{ch4}, is devoted to the alternative proof of Theorem \ref{th:theoremA}, relying on methods from spectral graph theory and random graphs. 

\section*{Acknowledgements}\label{acknowledgements}

We would like to thank Piotr Przytycki for suggesting to us the topic of paper, his help and valuable suggestions.

\section{Preliminary notions}\label{ch2}

\subsection{Models of random groups}\label{section:models}

In this section, we define relevant models of random groups. A comprehensive survey on the topic of random groups is given in \cite{ollivier2005}. All groups we consider will be finitely generated.

Suppose we have a finite set of generators $\Ss = \{s_1,\dots, s_n\}$. Let $\R$ be a set of words chosen randomly according to some probability distribution on the set of all words on the elements of $\Ss$ and their inverses. A random group will be a group given by the presentation $\pres{\Ss}{\R}$. Choosing a model of random groups consists of specifying a probability law for the set of relators $\R$. Unless stated otherwise, we consider only relators $r$ which are {\bf cyclically reduced}, i.e.\ if $r = s_1\dots s_l$, then $s_i \neq s_{i+1}^{-1}$ for $1 \leq i \leq l$ (where we identify $s_{l+1}$ and $s_1$). Usually, we will be interested in asymptotic properties of a random group when some parameter, e.g. number of generators $n$ or the length of relators $l$, goes to infinity. 

Whenever numbers like $(2n-1)^{ld}$, denoting an integer quantity (e.g. number of relators), appear, they are understood to mean their integer parts, e.g. $\lfloor{(2n-1)^{ld}} \rfloor$.

The most widely studied model of random groups is the {\it Gromov density model}.

\begin{definition}\label{def:gromov-density}
Fix a parameter $d \in (0,1)$, called the relation {\bf density}. A group $\Gamma$ in the {\bf Gromov density model $\gromov{n}{l}{d}$} is given by $\Gamma = \pres{\Ss}{\R}$, where $\vert\Ss\vert = n$ and $\R$ is a set of $(2n-1)^{ld}$ relators, each of them chosen uniformly and independently from the set of relators of length $l$ (note that the same relator can appear several times).
\end{definition}

\begin{definition}\label{def:gromov-density-whp}
Fix $n$ and $d$. We say that a property $P$ of groups holds {\bf with overwhelming probability} (w.o.p.) in the Gromov density model if we have $\lim\limits_{l \rightarrow \infty}\mathbb{P}(\Gamma \in \gromov{n}{l}{d}\ \mathrm{satisfies}\ P) = 1$.
\end{definition}

We will sometimes relax the requirement that exactly $(2n-1)^{ld}$ relators are chosen, i.e. we allow choosing $f(n) \approx (2n - 1)^{ld}$ relators, where ''$\approx$'' means that $\frac{f(n)}{(2n - 1)^{ld}} \to 1$ as $n \to \infty$. It will be clear that the properties we consider do not depend on such details.

As mentioned in the introduction, for $d \geq \frac{1}{2}$ groups in the Gromov density model are trivial or $\mathbb{Z}_{2}$ w.o.p., so we will only be interested in values of $d < \frac{1}{2}$.

Another model of random groups, introduced {\it ad hoc} in \cite{zuk}, is the {\it triangular model}.

\begin{definition}\label{def:triangular}
As before, fix $d \in (0,1)$. A group $\Gamma$ in the {\bf triangular model $\triangular{m}{d}$} is given by $\Gamma = \pres{\Ss}{\R}$, where $\vert\Ss\vert = m$ and $\R$ is a set of $(2m-1)^{3d}$ relators, chosen uniformly and independently from the set of relators of length 3.
\end{definition}

Intuitively, a random group in the triangular model $\triangular{m}{d}$ should be related to a Gromov random group $\gromov{n}{l}{d}$ with $2m = (2n)^{\frac{l}{3}}$, since then the numbers of relators, $(2m-1)^{3d}$ and $(2n-1)^{ld}$, are approximately the same. Indeed, we will see in Chapter \ref{ch3} that every group in $\gromov{n}{l}{d}$ has a subgroup of finite index which is a homomorphic image of a group in the (slightly modified) triangular model, with $m$ as above.

\begin{definition}\label{def:triangular-whp}
For fixed $d$, we say that a property $P$ of groups holds {\bf with overwhelming probability} (w.o.p.) in the triangular model if we have $\lim\limits_{m \rightarrow \infty}\mathbb{P}(\Gamma \in \triangular{m}{d}\ \mathrm{satisfies}\ P) = 1$.
\end{definition}

We note that if $d < \frac{1}{2}$, in both models w.o.p each relator appears at most once \cite{ollivier2005}:

\begin{remark}\label{rem:no-duplicates}
If $d < \frac{1}{2}$, w.o.p. all relators in the presentation of a random group in the Gromov model $\gromov{n}{l}{d}$ or the triangular model $\triangular{m}{d}$ are distinct words.
\end{remark}

The last model we shall use is the {\it permutation model}.

\begin{definition}\label{def:permutation}
Fix a set of generators $\Ss=\{s_1, \dots, s_n\}$ of size $n$. Choose $v \geq 1$ and pick uniformly and independently $v$ pairs of permutations $\{\pi_1^{1}, \pi_2^{1}\}, \dots, \{\pi_1^{v}, \pi_2^{v}\}$, where each $\pi_k^{i}, \ k=1,2$, is chosen randomly from the set of all permutations of $\{s_1, \dots, s_n, s_1^{-1}, \dots, s_n^{-1}\}$. A random group $\Gamma$ in the {\bf permutation model} $\permutation{n}{v}$ is $\Gamma = \pres{\Ss}{\R}$, with $\R$ consisting of $2nv$ words of the form $s_j^{\pm 1} \pi_1^{i}(s_j^{\pm 1})\pi_2^{i}(s_j^{\pm 1})$, for $i=1,\dots,v$ and $j=1\dots,n$.  
\end{definition}

Note that we place no restrictions on the permutations $\pi_{k}^{i}$, so the resulting words \\ $s_j^{\pm 1} \pi_1^{i}(s_j^{\pm 1})\pi_2^{i}(s_j^{\pm 1})$  might not be cyclically reduced. 

The permutation model is directly related to the following notion of a random graph:

\begin{definition}\label{def:configuration}
Fix a set of vertices $V=\{1, \dots, n\}$ and choose at random $v$ permutations $\pi_1, \dots, \pi_v$  of $V$. Define the set of (undirected) edges between vertices as $E=\{(i, \pi_k(i)) : i=1,\dots,n, \ k=1,\dots,v\}$. We allow multiple edges between two vertices. A random graph in the {\bf configuration model} $\permgraph{n}{v}$ is the graph $G=(V,E)$.
\end{definition}

A graph in the configuration model $\permgraph{n}{v}$ is a random regular graph of degree $2v$.  Such graphs have been extensively studied (see for example \cite{luczak}) and their spectral properties will be helpful in analysing random groups appearing in the permutation model.

\begin{definition}\label{def:permutation-whp}
For fixed $v$, we say that a property $P$ of groups holds {\bf with overwhelming probability} (w.o.p.) in the  permutation model if we have $\lim\limits_{n \rightarrow \infty}\mathbb{P}(\Gamma \in \permutation{n}{v}\ \mathrm{satisfies}\ P) = 1$.
\end{definition}

\subsection{Property (T)}\label{section:property-t}

We will now define {\it Kazhdan's property (T)} and formulate some basic facts. Here, we are only concerned with discrete, finitely generated groups --- for a more complete treatment, including property (T) for general topological groups, see \cite{bekka2008kazhdan}.

Let $\Gamma$ be a finitely generated group with a finite generating set $\Ss$. Consider $\mathcal{H}$, a Hilbert space, and $\pi: \Gamma \rightarrow U(\mathcal{H})$, a unitary representation of $\Gamma$ on $\mathcal{H}$. We will say that $\pi$ has {\bf almost invariant vectors} if for every $\varepsilon > 0$ there exists $u_{\varepsilon} \in \mathcal{H}$ such that for every $s \in \Ss$ $\norm{\pi(s)u_{\varepsilon} - u_{\varepsilon}} < \varepsilon \norm{u_{\varepsilon}}$ (call such $u_{\varepsilon}$ {\bf $\varepsilon$-invariant}). A vector $u \in \mathcal{H}$ is called {\bf invariant} if $\pi(g)u=u$ for every $g \in \Gamma$.

\begin{definition}\label{def:property-t}
We say that $\Gamma$ has {\bf property (T)} if for every $\mathcal{H}$ and $\pi$ the following holds: if $\pi$ has almost invariant vectors, then $\pi$ has an invariant vector. 
\end{definition}

Before we proceed to give examples, we note the following simple property:

\begin{remark}\label{th:homomorphism}
Let $\phi: \Gamma \rightarrow H$ be an epimorphism. Then if $\Gamma$ has property \emph{(T)}, so does $H$.
\end{remark}
\begin{proof}
Let $\Ss$ be a generating set of $\Gamma$. Take $\phi(\Ss)$ to be a generating set of $H$. Every representation $\pi$ of $H$ can be pulled back to a representation $\pi' = \pi \circ \phi$ of $\Gamma$. Note that $\varepsilon$-invariant vectors of $\pi$ are also $\varepsilon$-invariant of $\pi'$, so if $\pi$ has almost invariant vectors, so does $\pi'$. Since $\Gamma$ has property (T), $\pi'$ has an invariant vector, which is also an invariant vector of $\pi$.
\end{proof}

An easy consequence of the above fact is that property (T) is preserved under adding relators --- if $\Gamma = \pres{\Ss}{\R}$ has property (T), then $\Gamma' =\pres{\Ss}{\R \cup \R'}$, obtained from $\Gamma$ by adding an arbitrary set of relators $\R'$, also has property (T). Therefore, to prove that a group $\Gamma = \pres{\Ss}{\R}$ has property (T), it suffices to find a subset $\R' \subseteq \R$ of relators such that $\Gamma' = \pres{\Ss}{\R'}$ has property (T).

Another useful fact is that property (T) is preserved when passing to a subgroup of finite index \cite[Theorem 1.7.1]{bekka2008kazhdan}:

\begin{remark}\label{rem:finite-t}
Let $H$ be a subgroup of finite index in $\Gamma$. Then $H$ has property (T) if and only if $\Gamma$ has property (T).
\end{remark}

We illustrate the definition of property (T) with a few examples:
\begin{itemize}
\item $\mathbb{Z}$ does not have property (T), which can be seen by analysing its representation on $L^2(\mathbb{R})$ by translations
\item from Remark \ref{th:homomorphism}, any group admitting an epimorphism onto $\mathbb{Z}$ does not have property (T) --- this includes $\mathbb{Z}^n$ and free groups $F_n$
\item lattices in certain Lie groups, including $SL_n(\mathbb{R})$ for $n \geq 3$ and $Sp_{2n}(\mathbb{R})$ for $n \geq 2$, have property (T) \cite[Section 1.4, 1.5]{bekka2008kazhdan}
\end{itemize}

\subsection{Spectral criterion for property (T)}\label{section:spectral-t}

Given a group and its presentation, it is in general difficult to decide whether the group has property (T) or not. Here we describe, following \cite{zuk}, a simple sufficient condition for having property (T), based on spectral properties of a certain finite graph associated with the group's presentation. 

Let $G = (V, E)$ be a graph with the vertex set $V$ and the set of edges $E$ (we allow multiple edges between two vertices). For a vertex $v$ we denote its degree by $\deg(v)$. For functions $f, g : V \to \mathbb{R}$  their inner product is defined as:
\[\langle f, g \rangle = \sum\limits_{v \in V}f(v)g(v)\deg(v)\]
We can then introduce the discrete Laplacian $\Delta$:
\[
(\Delta f)(v) = f(v) - \frac{1}{\deg(v)}\sum\limits_{v' \sim v} f(v')
\]
where $v' \sim v$ means that vertices $v'$ and $v$ are adjacent (if there are multiple edges between $v$ and $v'$, then $v'$ appears in the sum with the corresponding multiplicity).

The Laplacian $\Delta$ is a linear operator from $\ell^2(G)$ to $\ell^2(G)$. It is easily seen that it is non-negative and self-adjoint with respect to the inner product defined above and therefore has real spectrum. 

Constant functions are mapped by $\Delta$ to $0$, so $0$ is always an eigenvalue. If $\Gamma$ is connected, then it is a simple eigenvalue. We will denote the smallest nonzero eigenvalue of $\Delta$ by $\lambda_{1}(\Gamma)$. 

The graph $L(\Ss)$ below is the same as the graph $L'(S)$ from \.Zuk's paper:

\begin{definition}[{\cite[Section 7.1]{zuk}}]\label{def:graph-l'}
Let $\Gamma = \pres{\Ss}{\R}$, where $\Ss = \{ s_{1}, \ldots, s_{n} \}$. We define the graph $L(\Ss)$ as follows. The vertices of this graph are generators $s_{1}, \ldots, s_{n}$ and their inverses $s_{1}^{-1}, \ldots, s_{n}^{-1}$. For each relator in $\R$ of the form $s_{x}s_{y}s_{z}$ we introduce edges $({s}_{x}s_{y}^{-1}), (s_{y}s_{z}^{-1}), (s_{z}s_{x}^{-1})$.
\end{definition}

This definition allows having multiple edges between two distinct vertices, but does not allow loops from a vertex to itself (since we assume that relators are cyclically reduced). Note that only relators of length $3$ (corresponding to triangles in the Cayley graph of $\Gamma$ associated with the presentation) are relevant to the definition of $L(\Ss)$.

We will always assume that $L(\Ss)$ is connected. As noted in \.Zuk's paper, even if $L(\Ss)$ is not connected (which is the case, for example, for $\mathbb{Z}^2$ with standard generators), we can change the generating set, replacing $\Ss$ by $\Ss \cup \Ss^2$, and it is easily checked that the latter set gives a connected graph. Thus, every group has a generating set $\Ss$ for which $L(\Ss)$ is connected.

Intuitively, for a presentation consisting only of relators of length $3$, the graph $L(\Ss)$ is the vertex link of the triangular presentation 2-complex.

We can now formulate the theorem connecting spectral properties of the Laplacian and property (T).

\begin{theorem}[{\cite[Proposition 6]{zuk}}]\label{th:lambda-1/2}
Let $\Gamma$ be a group given by a presentation $\pres{\Ss}{\R}$ and assume that $L(\Ss)$ is connected. If $\lambda_{1}(L(\Ss)) > \frac{1}{2}$, then $\Gamma$ has property \emph{(T)}.
\end{theorem}

The theorem first appears in \cite{ballman} (see also \cite{zuk-fr}). It is usually formulated for a version of $L(\Ss)$ without multiple edges.

Property (T) and the above criterion are closely connected to the spectral properties of $G$-invariant random walks and existence of certain Poincar{\'e} inequalities on the Cayley graph of $\Gamma$. A good reference for this topic is \cite{oll-aut-t} (see also \cite{bekka2008kazhdan}). 

The following example shows that the constant $\frac{1}{2}$ cannot be improved \cite{zuk}. Take $\Gamma = \mathbb{Z}^2$ with the generating set $\Ss = \{ (0,1), (1,0), (1,1) \}$. The graph $L(\Ss)$ is a cycle of six vertices and it can be checked that the eigenvalues of $\Delta$ on this graph are $0, \frac{1}{2}, \frac{3}{2}, 2$. We then have $\lambda_{1}(L(\Ss)) = \frac{1}{2}$, but $\mathbb{Z}^2$ does not satisfy property (T).

Later on it will be also convenient to decompose $L(\Ss)$ into subgraphs, as in the following definition. 

\begin{definition}
Let $\Gamma = \pres{\Ss}{\R}$ be a group. We define graphs $L_{i}$, $i = 1, 2, 3$, such that each of them has the same vertex set as $L(\Ss)$ and for each relator of the form $s_{x}s_{y}s_{z}$ we insert and edge $(s_{x}, s_{y}^{-1})$ into $L_{1}$, $(s_{y}, s_{z}^{-1})$ into $L_{2}$ and $(s_{z}, s_{x}^{-1})$ into $L_{3}$.
\end{definition}

\subsection{Proofs of main theorems -- outline}\label{section:zuk-outline}

The main results we are concerned with in this paper are the following theorems:

\begin{theoremalph}\label{th:theoremA}
For density $d > \frac{1}{3}$, a random group in the triangular model $\triangular{m}{d}$ has property \emph{(T)} with overwhelming probability.
\end{theoremalph}

\begin{theoremalph}\label{th:theoremB}
For density $d > \frac{1}{3}$, a random group in the Gromov model $\gromov{n}{l}{d}$ has property \emph{(T)} with overwhelming probability.
\end{theoremalph}

Here we give an outline of \.Zuk's original argument concerning Theorem \ref{th:theoremA}, which is contained in \cite{zuk}, and point out gaps in the proof which need to be filled. Then we sketch how this can be used to deduce Theorem \ref{th:theoremB}.

The first step is to prove that w.o.p. a random group in the permutation model $\permutation{n}{v}$ satisfies property (T) (for a suitably chosen value of $v$). 

This is accomplished by invoking the spectral criterion. Namely, let $\Gamma = \pres{\Ss}{\R}$ be a random group in the $\permutation{n}{v}$ model. Consider the graph $L(\Ss)$ and the associated subgraphs $L_{1}$, $L_{2}$, $L_{3}$. It is readily proved, as in Żuk's paper \cite[Lemma 6]{zuk}, that if $\lambda_{1}(L_{i}) > \frac{1}{2}$ for $i = 1, 2, 3$, then $\lambda_{1}(L(\Ss)) > \frac{1}{2}$. We can then invoke Theorem \ref{th:lambda-1/2} to conclude that $\Gamma$ has property (T). To prove that each of $L_{i}$, say $L_{1}$, has $\lambda_{1} > \frac{1}{2}$, we observe that it is exactly a random graph in the configuration model $\permgraph{n}{v}$ (Definition \ref{def:configuration}) and such graphs typically have large second Laplacian eigenvalue, thanks to the following theorem by Friedman:

\begin{theorem}[{\cite[Theorem B]{friedman}}]\label{th:friedman}
Let $L$ be a random graph in the configuration model $\permgraph{n}{v}$. Then 
\[
\lim\limits_{n \to \infty}\mathbb{P}\left( \lambda_{1}(L) > 1 - \left( \frac{\sqrt{2v - 1}}{v} + \frac{\log v}{v} + \frac{c}{v} \right) \right) = 1
\]
where $c$ is some constant independent of $v$. 
\end{theorem}

The inequality in \cite{friedman} is formulated in terms of eigenvalues of the adjacency matrix, but it is straightforward to rewrite it as an inequality for the Laplacian eigenvalues.

We fix $v$ to be large enough so that the theorem above gives $\lambda_{1}(L_{i}) > \frac{1}{2}$ (actually we can have $\lambda_{1}(L_{i})$ arbitrarily close to $1$, as $1 - \left( \frac{\sqrt{2v - 1}}{v} + \frac{\log v}{v} + \frac{c}{v}\right) \rightarrow 1$ with $v \to \infty$). This concludes the first step.

The second step consists of showing that w.o.p. a random group in the triangular model for $d > \frac{1}{3}$ is a quotient of a random group in the permutation model. Since we know that random groups in the latter model w.o.p. have property (T) and property (T) is preserved under taking quotients, it will follow that w.o.p. a random group in the triangular model has property (T).

To this end we need to show that a typical presentation of a triangular group contains a subset of relators corresponding to pairs of permutations, as in the definition of $\permutation{n}{v}$. In \cite{zuk} this fact is stated without proof. Actually, this is equivalent to finding a {\it perfect matching} in a certain hypergraph closely related to $L(\Ss)$, which we will define in Chapter \ref{ch3}. We will prove there, using a theorem proved by Johansson, Kahn and Vu in \cite{kahnvu}, that indeed such a matching exists with high probability (although we should stress that the result from \cite{kahnvu} is highly nontrivial). This  gives us Theorem \ref{th:theoremA}.

To obtain Theorem \ref{th:theoremB}, we need to pass from the triangular model to the Gromov density model. As explained in \cite{ollivier2005}, the triangular model is typically ``less quotiented'' than the Gromov model, so we expect that groups in the latter model with high probability will be quotients of groups in the former. This is indeed the case, although only some of the details are provided in \cite{ollivier2005}. In Section \ref{triang2gromov} we prove the required result by constructing a homomorphism from a suitable modification of the triangular model to the density model (the modifications are needed to deal with technical issues involving nonreduced words).

This strategy of proof may be considered roundabout --- we are forced to look for a perfect matching in a random hypergraph, but only because of good spectral properties of random groups in the permutation model, i.e. we are only interested in the large eigenvalue the matching provides and not the matching itself. A more direct proof, avoiding combinatorial difficulties with finding matchings, might be desirable. We give such a proof in Section \ref{ch4}, where we prove directly in the triangular model (using known results about another model of random graphs and without resorting to permutations) that graphs $L_i$ typically have large $\lambda_1$.

\section{Proofs of main theorems - the triangular model and the Gromov model}\label{ch3}

\subsection{Basic notions of random graphs}\label{randomGraphs}

In this section we recall basic notions of random graph theory and introduce auxilliary definitions. We shall use standard asymptotic notation. For functions $f, g$, $f = O(g)$ means that $f(n) \leq C g(n)$ for some constant $C$, and similarly $f = \Omega(g)$ means that $f(n) \geq C g(n)$ for some $C$. $f = o(g)$ means that $\frac{f(n)}{g(n)} \to 0$ as $n \to \infty$ and $f = \omega(g)$ means that $\frac{f(n)}{g(n)} \to \infty$ as $n \to \infty$. Finally $f = \Theta(g(n))$ if both $f = O(g)$ and $f = \Omega(g)$ hold.

Fix a vertex set $V$ of size $n$. A {\bf $k$-hypergraph} $G=(V,E)$ with vertex set $V$ consists of a family of {\bf hyperedges} $E$, where each hyperedge $e \in E$ is a subset of $V$ of size $k$ (we allow repeated hyperedges). The case $k=2$ gives the usual notion of graph with edges. A hypergraph will be called $k$-partite if the vertex set $V$ can be partitioned into $k$ disjoint subsets, $V = V_1 \cup \dots \cup V_k$, such that each hyperedge is of the form $e=\{v_1, \dots, v_k\}$, $v_i \in V_i$ for $i=1, \dots, k$. A {\bf perfect matching} in a $k$-partite hypergraph is a set of disjoint edges that cover $V$.

One of the most widely used models of random graphs is the $G(n, p)$ model (also called the Erd\H{o}s-Renyi model), in which a random graph on $n$ vertices is obtained by taking each possible edge independently with probability $p$. The expected number of edges is $N=\binom{n}{2}p$. Typically we are interested in asymptotic properties of such graphs, when $n \to \infty$ and $p = p(n)$ depends on $n$. For given $p$, we say that a random graph has property $P$ {\bf with high probability} (w.h.p.) if $\lim\limits_{n \rightarrow \infty}\mathbb{P}(G \in G(n,p)\ \mathrm{satisfies}\ P) = 1$. A common phenomenon in $G(n,p)$ model is the existence of sharp thresholds, where for a given graph property $P$ there exists a threshold probability $p(n)$ such that $\lim\limits_{n \rightarrow \infty}\mathbb{P}(G \in G(n,p)\ \mathrm{satisfies}\ P) = 0$ if $p = o(p(n))$, but $\lim\limits_{n \rightarrow \infty}\mathbb{P}(G \in G(n,p)\ \mathrm{satisfies}\ P) = 1$ for $p = \omega(p(n))$.

A related model is the $G(n, M)$ model, where we choose uniformly at random a set of edges from all possible sets of edges of size $M$ on $n$ vertices. This is somewhat similar to the $G(n, p)$ model with $p$ such that $M=\binom{n}{2}p$ (so that the expected number of edges agrees) and sometimes results about one model can be transferred to the other \cite{luczak}. Both definitions are straightforwardly generalized to the case of hypergraphs. 

For technical reasons it will be also convenient to define two slightly different models of random hypergraphs. The motivation for introducing them will be clear from the next section.

\begin{definition}
Fix $n$ and $M$. Let $V$ be a set of size $2n$ whose elements are labeled by $s_{1}, \ldots, s_{n}, s^{-1}_{1}, \ldots, s^{-1}_{n}$. A random graph $G$ in the $\tripartite{n}{M}$ is a $3$-partite hypergraph obtained in the following way. The vertex set of $G$ consists of three disjoint copies $V_1$, $V_2$, $V_3$ of $V$. We choose uniformly and independently a set of $M$ hyperedges, where each hyperedge is a triple $(s_{x}, \, s_{y}, \, s_{z})$, $s_{x} \in V_{1}, s_{y} \in V_{2}, s_{z} \in V_{3}$.
\end{definition}

Let $e$ = $(s_{x}, \, s_{y}, \, s_{z})$ be an edge as in the previous definition. We will call $e$ a {\bf reduced edge} if the corresponding word $s_{x}s_{y}s_{z}$ is cyclically reduced. A {\bf reduced perfect matching} is a perfect matching whose all hyperedges are reduced.

\begin{definition}
Fix $n$ and $M$. A random graph $G$ in the $\tripartiter{n}{M}$ model is a $3$-partite hypergraph obtained in the same way as in the $\tripartite{n}{M}$ model, with the restriction that we choose edges only from the set of reduced edges. The notion of satisfying a property with high probability is defined in the same way as in the previous models.
\end{definition}

\subsection{From the permutation model to the triangular model}\label{perm2triang}

In this section we will describe how to pass from the permutation model to the triangular model --- roughly, we will prove that a random triangular group is a quotient of a random group in the permutation model. Our first step will be establishing the existence of a perfect matching in a random hypergraph associated with a random group in the $\triangular{m}{d}$ model. Then, after a few technical steps, we will apply this analysis to the $\triangular{m}{d}$ triangular model and a variant of the $\permutation{n}{v}$ permutation model. This will enable us to prove Theorem \ref{th:theoremA}.

Recall that a random group $\Gamma = \pres{\Ss}{\R}$ in the triangular model $\triangular{m}{d}$ is obtained by choosing independently $(2m-1)^{3d}$ relators of the form $s_{x}s_{y}s_{z}$. From now on we take $d > \frac{1}{3}$ so that $3d = 1 + \varepsilon$, $\varepsilon > 0$. As we want to pass from this model to the permutation model, we focus our attention on finding a subset of relators of the form $s^{\pm 1}\pi_1(s^{\pm 1})\pi_2(s^{\pm 1})$ for $s \in \Ss$, where $\pi_i$ is a permutation. Each $s$ should appear in these relators three times (playing the role of $s$, $\pi_1(s')$ and $\pi_2(s'')$ for some $s', s''$). 

The requirement that $\R$ contains a subset of relators of the desired form is naturally rephrased in terms of the following $3$-hypergraph associated to the presentation.

\begin{definition}
Let $\Gamma = \pres{\Ss}{\R}$ be a group in the $\triangular{m}{d}$ model. We define a graph $\hypergraph{\Ss}$ in the following way: take the vertex set to be $V_1 \cup V_2 \cup V_3$, where $V_i$ are disjoint copies of $\Ss \cup \Ss^{-1}$. For a relator $s_{x}s_{y}s_{z} \in \R$ insert a hyperedge $(s_x, s_y, s_z), s_x \in V_1, s_y \in V_2, s_z \in V_3$ (we remove duplicate edges).
\end{definition}

Choosing a random group in the triangular model $\triangular{m}{d}$ gives us a random graph $\hypergraph{\Ss}$ in the $\tripartiter{m}{M}$ model for $M = (2m - 1)^{3d} = (2m - 1)^{1 + \varepsilon}$ (note that both models allow only reduced relators and edges, respectively). To find a subset of relators of the form $s^{\pm 1}\pi_1(s^{\pm 1})\pi_2(s^{\pm 1})$, for $s \in \Ss$, it is enough to find a perfect matching in $\hypergraph{\Ss}$.

In principle, models $\tripartite{n}{M}$ and $\tripartiter{n}{M}$ allow duplicate edges. However, we will use these graphs in the setting of random groups with density $< \frac{1}{2}$, where with overwhelming probability there are no duplicate relators (Remark \ref{rem:no-duplicates}). We can therefore assume that no duplicate edges will occur in random graphs described above. 

The key result upon which our proof is based is the following result from random graph theory \cite[Corollary 2.6]{kahnvu}:

\begin{theorem}\label{th:kahn}
Let $p(n) = \Omega(\frac{\log n}{n^2})$. Then a random 3-partite 3-hypergraph in $G_3(n, M)$ for $M = \Omega(n^3 p)$ contains a perfect matching with high probability.
\end{theorem}

The theorem in \cite{kahnvu} is actually proved for ordinary, not $k$-partite, hypergraphs, but the proof requires essentially no modifications to acommodate this change.

Theorem \ref{th:kahn} is a special case of a more general theorem in \cite{kahnvu}, concerning thresholds for existence of {\it balanced factors} in random graphs. In particular, it solves the long-standing ``Shamir problem'', asking for the threshold probability for existence of a perfect matching in a hypergraph. The proof is rather difficult, employing concentration inequalities and entropy estimates as well as intricate combinatorial arguments.

This is in sharp contrast with the case of bipartite graphs, where a simple argument based on Hall's marriage theorem gives a threshold probability of $p=\frac{\log n}{n}$ \cite[Corollary 7.13]{bollobas2001random}. The lack of a Hall-type theorem for hypergraphs, giving necessary and sufficient conditions for a perfect matching, makes generalizing that argument to hypergraphs unfeasible. Indeed, the problem of finding perfect matchings even in 3-hypergraphs is NP-hard, so it is unlikely that a simple Hall-type criterion exists. There are several partial criteria for matchings in hypergraphs (\cite{haxell}, \cite{haxell2}), but none of them seems sufficient to establish our result.

Easier arguments can be given \cite[Theorem 4.19]{luczak} to obtain a matching as soon as $M n^{-\frac{4}{3}} \rightarrow \infty$. One first conditions the graph on the sequence of degrees of its vertices and then applies the second moment method. However, this is too weak for our purposes, since it would only give Theorem \ref{th:theoremA} for densities $d > \frac{4}{9}$.

As an immediate corollary of Theorem \ref{th:kahn}, we obtain:

\begin{corollary}\label{th:nonred-matching}
With high probability a random graph in the $\tripartite{m}{M}$ model for $M = (2m - 1)^{1 + \varepsilon}$ contains a perfect matching.
\end{corollary}

Since we are interested in finding a matching in the reduced graph $\tripartiter{m}{M}$, we first need to establish the following fact:

\begin{corollary}\label{lm:derangements}
A random graph in the $\tripartite{m}{M}$ model, for $M = (2m - 1)^{1 + \varepsilon}$, contains a reduced perfect matching with high probability.
\end{corollary}

\begin{proof}
Let $G$ be a random graph in $\tripartite{m}{M}$ model. A perfect matching corresponds to a pair of permutations $\pi_{1}$, $\pi_{2}$ such that the edges of the matching are of the form $(s, \pi_{1}(s), \pi_{2}(s))$, $s \in \Ss \cup \Ss^{-1}$. Being a reduced matching imposes the additional conditions $s \neq \pi_{1}(s)^{-1}$, $\pi_{1}(s) \neq \pi_{2}(s)^{-1}$ and $\pi_{2}(s) \neq s^{-1}$, for each $s$. Therefore reduced perfect matchings correspond to pairs $\pi_{1}$, $\pi_{2}$ such that, after composing with an involution, $\pi_{1}$ is a {\it derangement} (a permutation without a fixed point), and, once $\pi_{1}$ is chosen, $\pi_2$ is subject to the constraints $\pi_2(s) \neq s^{-1}, \pi_1(s) \neq \pi_2(s)^{-1}$, which means that for each $s$ it has now to avoid two values.

It is well known that as $m \rightarrow \infty$, probability that a random permutation is a derangement approaches $\frac{1}{e}$. Furthermore, it can be shown that permutations with two forbidden values, as in the case of $\pi_2$, also asymptotically, as $m \rightarrow \infty$, form a constant fraction of all permutations (this is equivalent to the so called {\it m\'enage problem} and the constant is $\frac{1}{e^2}$ \cite{menage}). Therefore, by taking $m$ large enough, we can assume that for a random pair of permutations $\pi_{1}$, $\pi_{2}$ the probability of satisfying the conditions for a reduced matching is bounded from below by a constant $p$ independent of $m$.

In our model, the number of edges, of order $m^{1 + \varepsilon}$, is asymptotically greater than the threshold value from Theorem \ref{th:kahn}. Since the threshold number of edges is only $m \log m$, we can easily partition the edge set into $f(m)=C\frac{m^{\varepsilon}}{\log m}$ sets of size $m\log m$, $f(m) \rightarrow +\infty$. In each such set we have a perfect matching with high probability, and since probability that a perfect matching is reduced is $\geq p$, we have a perfect reduced matching with some probability $p' > 0$, independent of $m$. The probability that there is no perfect reduced matching is therefore less than $(1-p')^{f(m)}$, which goes to $0$ as $m \rightarrow \infty$.

\end{proof}

Now we are ready to find a perfect matching in the reduced graph model:

\begin{corollary}\label{th:reduced-matching}
A random graph in the $\tripartiter{m}{M}$ model, for $M = (2m - 1)^{1 + \varepsilon}$, contains a (necessarily reduced) perfect matching with high probability.
\end{corollary}

\begin{proof}
From Corollary \ref{lm:derangements} it follows that a random graph in the $\tripartite{m}{M}$ model, for $M = (2m - 1)^{1 + \varepsilon}$, contains a reduced perfect matching with high probability. Let $\mathbb{P}_M(F)$ denote the probability that a random graph in $\tripartite{m}{M}$ contains such a matching, $\mathbb{P}^{red}_M(F)$ --- the same probability for $\tripartiter{m}{M}$ and let $E_{M'}$ be the event that a random graph in $\tripartite{m}{M}$ contains exactly $M'$ reduced edges, $M' \leq M$.

Now, consider the following procedure --- we choose a random graph $G$ in the $\tripartite{m}{M}$ model and remove nonreduced edges, obtaining a graph $G'$. Fixing any $M' \leq M$ and conditioning on existence of exactly $M'$ reduced edges in $G$ gives us $G'$ that is a random graph in $\tripartiter{m}{M'}$. $G$ contains a reduced matching if and only if $G'$ does. Since $M'$ is now fixed, the distribution of nonreduced edges in $G$ is independent of the existence of a reduced matching, so $\mathbb{P}_M(F|E_{M'}) = \mathbb{P}^{red}_{M'}(F)$. $G'$ is a random graph in $\tripartiter{m}{M'}$ and $M' \leq M$, so $\mathbb{P}^{red}_{M'}(F) \leq \mathbb{P}^{red}_{M}(F)$. Thus, we obtain $\mathbb{P}_M(F|E_{M'}) \leq \mathbb{P}^{red}_{M}(F)$ and since this is true for every $M'$, upon taking the total distribution $\mathbb{P}_M(F)$ we obtain $\mathbb{P}_M(F) \leq \mathbb{P}^{red}_{M}(F)$. Therefore a random graph in the $\tripartiter{m}{M}$ model contains a perfect matching with high probability.
\end{proof}

We can now formally state and prove the main result of this section. Since relators in the triangular model are reduced, it is necessary to replace the usual permutation model $\permutation{n}{v}$ with the reduced model $\permutationr{n}{v}$:

\begin{definition}
The {\bf reduced permutation model}, denoted by $\permutationr{n}{v}$, is the random group model defined analogously to $\permutation{n}{v}$ (Definition \ref{def:permutation}), but we allow only relators which are cyclically reduced, i.e. only permutations $\{\pi_1, \pi_2\}$ with $\pi_1(s) \neq s^{-1}, \pi_1(s) \neq \pi_2(s)^{-1}, \pi_2(s) \neq s^{-1}$ for every $s \in \Ss \cup \Ss^{-1}$  are admissible. 
\end{definition}

\begin{definition}\label{def:configuration-red}
The random graph model $\permrgraph{n}{v}$ (compare to Definition \ref{def:configuration}) is a random graph on a vertex set $V, |V|=2n,$ labelled by $\{s_1, \dots, s_n, s_1^{-1}, \dots, s_n^{-1}\}$, which is constructed by choosing at random $v$ permutations $\pi_1, \dots, \pi_v$ for which $\pi_k(s) \neq s^{-1}$ for every $s \in V$. The edge set $E$ is defined as $E = \{(s, \pi_k(s)) : s \in V, k=1,\dots, v\}$.
\end{definition}

For $\permutationr{n}{v}$, it is easy to see that if we define graphs $L(S)$ and $L_1, L_2, L_3$ as in Section \ref{section:zuk-outline}, each $L_i$ will be a random graph in $\permrgraph{n}{v}$.

We know from Theorem \ref{th:friedman} that for $v$ large enough a random graph $L$ in $\permgraph{n}{v}$ has $\lambda_1(L)> \frac{1}{2}$ with high probability. The same holds for $\permrgraph{n}{v}$:

\begin{lemma}\label{lm:friedman-reduced}
For $v$ large enough, a random graph $L$ in $\permrgraph{n}{v}$ is connected and has $\lambda_1(L)> \frac{1}{2}$ with high probability. 
\end{lemma}

\begin{proof}
As in the proof of Lemma \ref{lm:derangements}, permutations such that $\pi(s) \neq s$ for all $s$ (derangements) asymptotically, as $m \rightarrow \infty$, form a constant fraction of all permutations, with constant $\frac{1}{e}$. If we pick $v$ permutations $\pi_1, \dots, \pi_v$ and require $\pi_k(s) \neq s$ for $k=1, \dots, v$ and all $s$, the probability that all are derangements also, as $m \rightarrow \infty$, approaches a constant. Choosing $v$ permutations with these restrictions corresponds to graph model $\permrgraph{m}{v}$, while choosing unrestricted permutations corresponds to $\permgraph{m}{v}$. From Theorem \ref{th:friedman}, if we choose $v$ large enough, a random graph $L$ in $\permgraph{m}{v}$ has $\lambda_1 > \frac{1}{2}$ with high probability. Furthermore, it is well known that random regular graphs such as $L$ are connected (\cite[Theorem 9.20]{luczak}). Since the probability that $L$ conforms to restrictions of $\permrgraph{m}{v}$ is bounded from below by a constant, independent of $n$, graphs in $\permrgraph{m}{v}$ will also have $\lambda_1 > \frac{1}{2}$ with high probability.
\end{proof}

From Lemma \ref{lm:friedman-reduced} and and Theorem \ref{th:lambda-1/2}, we obtain:

\begin{corollary}\label{cor:permutationr-t}
For $v$ sufficiently large, a random group in the reduced permutation model $\permutationr{n}{v}$ has property (T) with overwhelming probability.
\end{corollary}

We therefore obtain the main result of this section, Theorem \ref{th:theoremA}.

\begin{proof}[Proof of Theorem \ref{th:theoremA}]
Since $d > \frac{1}{3}$, we have $(2m-1)^{1 + \varepsilon}$ relators and thus from Corollary \ref{th:reduced-matching} a random triangular presentation contains, with overwhelming probability, a subset of relators of the form $s^{\pm 1}\pi_1(s^{\pm 1})\pi_2(s^{\pm 1})$ for $s \in \Ss$. Actually, for any fixed $v$ we can find $v$ disjoint subsets of such relators --- it is enough to divide the set of relators $\R$ arbitrarily into $v$ equal disjoint parts and find a desired subset in each of them. Obviously if we can find such a subset with overwhelming probability in one part, we will also be able to do this in all $v$ parts, as $v$ is fixed and each part has, up to a constant, the same number of relators as $\R$. 

Thus, with overwhelming probability we can assign to a random group $\Gamma$ in $\triangular{m}{d}$ a group $\Gamma'$ in $\permutationr{m}{v}$ such that $\Gamma$ is a quotient of $\Gamma'$. The choice of $v$ subsets of relators is not unique, but we can perform it in such a way that $\Gamma'$ is uniformly random, i.e. is a random group in $\permutationr{m}{v}$. For example, we can first enumerate such subsets in arbitrary order and always choose the first one that appears. If we take $v$ large enough, from Corollary \ref{cor:permutationr-t} $\Gamma'$ has property (T) with overwhelming probability, and since $\Gamma$ is a quotient of $\Gamma'$, $\Gamma$ has propery (T) with overwhelming probability.
\qedhere
\end{proof}

\subsection{From the triangular model to the Gromov model}\label{triang2gromov}
 
In this section we show how to pass from the triangular model to the Gromov density model and thus complete the proof of Theorem \ref{th:theoremB}. The idea is, as before, to prove that a random group in the Gromov model is a quotient of a random group in the triangular model. We first describe the proof informally and then present a more detailed argument, dealing with certain technical issues.

The first attempt, based on \cite[Section I.3.g]{ollivier2005},  would be as follows. Let us fix $n$, $l$, $d$ and let $W$ denote the set of all reduced words over letters $a_{1}^{\pm 1}, \ldots, a_{n}^{\pm 1}$ of length $l/3$ (assuming that $l$ is divisible by $3$). If we now consider a group $\Gamma$ in the $\triangular{m}{d}$ model with $2m = |W|$, we see that there is a natural mapping from $\triangular{m}{d}$ to $\gromov{n}{l}{d}$ (with $a_{1}, \ldots, a_{n}$ as generators in the latter model), with the same $d$. Namely, let $W^{+}, W^{-}$ be such that $W = W^{+} \cup W^{-}$, $|W^{+}| = |W^{-}| = \frac{1}{2}|W|$, $w_{i} \neq w_{j}^{-1}$ for $w_{i}, w_{j} \in W^{+}$ and words from $W^{-}$ are inverses of words from $W^{+}$. We then enumerate the elements of $W^{+}$ arbitrarily and define the homomorphism as $\phi(s_{i}) = w_{i}$, where $s_{1}, \ldots, s_{m}$ are generators in the triangular model and $w_{1}, \ldots, w_{m}$ are distinct elements of $W^{+}$. This is a homomorphism from $\Gamma$ into a group $\Gamma'$ in the $\gromov{n}{l}{d}$ model, as each relator of the form $s_{x}s_{y}s_{z}$ is mapped to a word of length $l$, a valid relator in the Gromov model (with the same density $d$, as $(2m - 1)^{3d} = (|W| - 1)^{3d} \approx (2n -1)^{ld}$). In fact, this is an epimorphism onto a subgroup of finite index in $\Gamma'$ and since every random group in the Gromov model arises from a triangular group by means of this homomorphism (so that the probability distribution induced on $\triangular{m}{d}$ is uniform), it follows that with high probability groups in $\gromov{n}{l}{d}$ have property (T).

The problem with this construction is that the words $\phi(s_{x}s_{y}s_{z})$ obtained in the image of the homomorphism may not be reduced. For example, if $\phi(s_{1}) = a_{1}a_{2}a_{3}$ and $\phi(s_{2}) = a_{3}^{-1}a_{1}a_{1}$, then $s_{1}s_{2}s_{2}$, a valid relation in the triangular model, is mapped to $a_{1}a_{2}a_{3}a_{3}^{-1}a_{1}a_{1}a_{3}^{-1}a_{1}a_{1}$, which is not reduced and hence not a valid relator in the Gromov model. As such reductions will occur with overwhelming probability, we have to introduce a modification of the triangular model and the homomorphism to deal with this issue.

\begin{definition}\label{def:triangular-positive}
Fix $d \in (0,1)$. A group $\Gamma$ in the {\bf positive triangular model $\triangularplus{m}{d}$} is given by $\Gamma = \pres{\Ss}{\R}$, where $\vert\Ss\vert = m$ and $\R$ is a set of $(2m-1)^{3d}$ relators, chosen independently and uniformly from the set of relators of length $3$, with the restriction that only words not using elements from $\Ss^{-1}$ are chosen.
\end{definition}

The graph $L(\Ss)$ associated with the group is defined in the same way as for the standard triangular model. Note that this time $L(\Ss)$ is always a bipartite graph, with vertex sets $\Ss$ and $\Ss^{-1}$, as there are no edges of the form $ss'$ or $s^{-1}s'^{-1}$.

We would like to prove the analogue of Theorem \ref{th:theoremA} for the positive triangular model.

The main difference with the standard triangular model is that we cannot use Theorem \ref{th:friedman} to deduce large $\lambda_1$ eigenvalue of the graph induced by the permutation subset of relators. Theorem \ref{th:friedman} works for graphs chosen at random among all $v$-regular graphs, and in our case each graph $L_{i}$, as defined in Section \ref{section:zuk-outline}, is always bipartite and $v$-regular. Fortunately, the following analogue of Theorem \ref{th:friedman} easily follows from a more general result about random graphs (which, incidentally, can also be used to prove Theorem \ref{th:friedman}, with a different bound):

\begin{theorem}[{\cite[Theorem 1.2]{friedman-relativeexpanders}}]\label{th:friedman-bi}
Consider the vertex set $V$ consisting of two disjoint copies of $\{1, \ldots, n\}$. Let $G$ be a random bipartite $v$-regular graph obtained by choosing uniformly at random $v$ permutations $\pi_{1}, \ldots, \pi_{v}$ of $\{1, \ldots, n\}$ and inserting edges $\left(i, \pi_{j}(i)\right)$, $j = 1, \ldots, v$, $i = 1, \ldots, n$, into $V$. Then for any $\varepsilon > 0$:
\[
\lim\limits_{n \to \infty}\mathbb{P}\left( \lambda_{1}(G) \geq 1 - \left( \frac{\sqrt{2v}(v - 1)^{\frac{1}{4}}}{v} + \frac{\varepsilon}{v} \right)  \right) = 1
\]
\end{theorem}

This in particular implies that $\lambda_1 \to 1$ with $v \to \infty$.

One obtains this theorem by applying \cite[Theorem 1.2]{friedman-relativeexpanders} to a base graph consisting of two vertices connected by $v$ edges. Every bipartite $v$-regular graph is then a cover of this base graph.

\begin{theorem}\label{th:triangular-t-pos}
A random group in the $\triangularplus{m}{d}$ model for $d > \frac{1}{3}$ has property (T) with overwhelming probability
\end{theorem}

\begin{proof}
The proof is esentially the same as the proof of Theorem \ref{th:theoremA}, so we will not repeat all the details. In fact, it is even easier, as we do not have to deal with possible reductions and require that the permutations obtained in the proof are derangements.

We define the graph $G_{3}(n, M)$ in the same way as before, with the modification that it has only vertices labelled by $s_{1}, \ldots, s_{n}$, as their inverses are not used in any relators in $\triangularplus{m}{d}$. Using Theorem \ref{th:kahn}, we deduce that with high probability it has a perfect matching, which corresponds to a set of relators of the form $s\pi_{1}(s)\pi_{2}(s)$. This enables us to show that with high probability every group in the $\triangularplus{m}{d}$ model is a quotient of a random group in the modified $\permutation{n}{v}$ model, where we disallow inverses of generators to appear, as in the definition of $\triangularplus{m}{d}$. By Theorem \ref{th:friedman-bi} the latter groups satisfy the spectral criterion, as in the case of standard permutation model, and the theorem follows.
\end{proof}

We now proceed to define the desired homomorphism. Fix $n$, $l$, $d$ and a set of generators $\mathcal{A} = \{a_{1}, \ldots, a_{n}\}$. We will call elements from $\mathcal{A}$ positive letters and elements from $\mathcal{A}^{-1}$ negative letters. Let $W_{l/3}$ denote the set of all reduced words over $\mathcal{A} \cup \mathcal{A}^{-1}$ of length $l/3$ which begin and end with a positive letter. By $W'_{l}$ we denote the set of words of the form $w_{1}w_{2}w_{3}$, where $w_{i} \in W_{l/3}$ (all words in $W'_{l}$ are reduced and of length $l$). Consider now the set $\Ss = \{ s_{1}, \ldots, s_{m} \}$, where $m = |W_{l/3}|$, and enumerate all elements of $W_{l/3}$ in an arbitrary way. Let $\langle g_{1}, \ldots, g_{m}\rangle$ denote the free group generated by $\{g_{1}, \ldots, g_{m}\}$. We define a homomorphism $\phi : \langle s_{1}, \ldots, s_{m}\rangle \to \langle a_{1}, \ldots, a_{n} \rangle$ by $\phi(s_{i}) = w_{i}$, where $w_{i}$ is the $i$-th element of $W_{l/3}$.

Now note that if $\Gamma = \pres{\Ss}{\R}$ is a group in the $\triangularplus{m}{d}$ model, then each relator $s_{x}s_{y}s_{z}$ is mapped by $\phi$ to a word in $W'_{l}$. Therefore $\phi$ projects to a homomorphism (for simplicity also denoted by $\phi$) from $\Gamma$ to $\Gamma' = \pres{\mathcal{A}}{\phi(\R)}$, where $|\R| = |\phi(\R)| = (2m -1)^{3d}$ and each relator in $\phi(\R)$ is from $W'_{l}$. Actually it is an epimorphism onto a subgroup of finite index in $\Gamma'$.

\begin{lemma}\label{lm:finite-index}
$\phi(\Gamma)$ is a subgroup of finite index in $\Gamma'$.
\end{lemma}

\begin{proof}
Throughout the proof $a, b, c$ will always denote positive letters. 

$\phi(\Gamma)$ is generated by all reduced words of length $l/3$ which are of the form $avb$. We will show that every word over $\mathcal{A} \cup \mathcal{A}^{-1}$ is equal in the free group to a word of the form $u$, $tu$ or $stu$, where $u \in \phi(\Gamma)$ and $s, t \in \mathcal{A} \cup \mathcal{A}^{-1}$, which is enough to prove that $\phi(\Gamma)$ has finite index in $\Gamma'$.

Observe that all words of the form $avb^{-1}$ (and similarly $a^{-1}vb$), where $v$ has even length, belong to $\phi(\Gamma)$. Namely, set $k = l/3 - 2$ and let $v$ have even length. If $v$ has length at most $2k$, then we write $v = v'v''$, where $v'$, $v''$ have equal length, and concatenate words $av'ut$ and $t^{-1}u^{-1}v''b^{-1} = (bv''^{-1}ut)^{-1}$, where $u$ is a word and $t$ is a positive letter, chosen so that both $av'ut$ and $bv''^{-1}ut$ are reduced. If $v$ has length greater than $2k$, then we write $v$ as $v = v_{1}v_{2}\ldots v_{j} v'v''$, for some $j$, where each $v_{i}$ has length $2k$ and $v'$, $v''$ have length at most $k$. We then concatenate words $av_{1}t_{1}^{-1}$, $t_{1}v_{2}t_{2}^{-1}$, $\ldots$, $t_{j-1}v_{j}t_{j}^{-1}, t_{j}v'ut_{j+1}$ and $t_{j+1}^{-1}u^{-1}v''b^{-1}$, where again $u$ and $t_{i}$ are chosen so that all these words are reduced.

Suppose now that we have a reduced word $w$ over $\mathcal{A} \cup \mathcal{A}^{-1}$. Without loss of generality we can assume that $w$ begins with $a$ or begins with $a^{-1}$ and ends with a positive letter. Assume that $w$ begins with $a$ (the other case is dealt with analogously). Consider first the case when $w$ has odd length. If $w$ ends with a positive letter, say $b$, then for any $c \neq a$ the word $c^{-1}w$ is of the form $c^{-1}w'b$, where $w'$ has even length, so it belongs to $\phi(\Gamma)$. Similarly, if $w$ ends with a negative letter, say $b^{-1}$, than for $c$ the word $cw$ belongs to $\phi(\Gamma)$.

If $w$ has even length, then we have two cases. If it ends with some $b^{-1}$, then it is of the form $aw'b^{-1}$, where $w'$ has even length, so it belongs to $\phi(\Gamma)$. If it ends with some $b$, then for $c \neq a$ we have that $c^{-1}a^{-1}w$ is of the form $c^{-1}w'b$, where $w'$ has even length, so $c^{-1}a^{-1}w$ belongs to $\phi(\Gamma)$.
\end{proof}

We are now in the position to finish the proof of our main theorem.

\begin{proof}[Proof of Theorem \ref{th:theoremB}]
Fix $n$, $l$ and $1/3 < d < 1/2$. Let $m = |W_{l/3}|$ (with all the notation as above). We will first show that if $\Gamma' = \pres{\mathcal{A}}{\R}$ is a random group in the $\gromov{n}{l}{d}$ model, but with the restriction that we choose $(2m - 1)^{3d}$ relators only from $W'_{l}$, then with overwhelming probability $\Gamma'$ satisfies property (T).

Suppose $\Gamma'$ is such a group. By means of the construction described above, we can find a group $\Gamma$ in the $\triangularplus{m}{d}$ model, with the set of generators $\Ss$ and $m = |W_{l/3}|$, such that $\phi(\Gamma)$ is a subgroup of finite index in $\Gamma'$ (Lemma \ref{lm:finite-index}). The probability distribution on $\triangularplus{m}{d}$ obtained in this way is uniform and from Theorem \ref{th:triangular-t-pos} we know that with overwhelming probability $\Gamma$ has property (T). Since $\phi$ is an epimorphism from $\Gamma$ onto a subgroup of finite index in $\Gamma'$, from Remark \ref{th:homomorphism} and Remark \ref{rem:finite-t} we have that $\Gamma'$ satisfies property (T) with overwhelming probability.

Now let $\Gamma' = \pres{\Ss}{\R}$ be a random group in the $\gromov{n}{l}{d}$ model. We will show that with high probability $\R$ contains a subset $\R'$ of at least $(2m - 1)^{3d}$ relators from $W'_{l}$. 

For $C > 0$, where $C \approx \frac{1}{2^6}$, we have $|W'_{l}| = \frac{1}{C} (2n - 1)^{l}$. This means that the average number of relators from $W'_{l}$ in $\R$ is $\frac{1}{C} (2n - 1)^{ld}$. With overwhelming probability $\R$ contains at least, say, $\frac{1}{2}\frac{1}{C}(2n - 1)^{ld}$ such relators. Since we want to have $(2m - 1)^{3d} = (|W_{l/3}| - 1)^{3d} \approx |W'_{l}|^{d} = \frac{1}{C} (2n - 1)^{ld}$ relators, we should make an inessential change in the definition of $\triangularplus{m}{d}$, taking $\frac{1}{2}(2m - 1)^{3d}$ relators instead of $(2m - 1)^{3d}$. Alternatively we could take the density in the triangular model to be some $d'$ arbitrarily smaller than $d$ and construct the homomorphism from $\triangularplus{m}{d'}$ to $\gromov{n}{l}{d}$. Another option is choosing relators of length in $[l-C,l+C]$ for some constant $C>0$.

Since with overwhelming probability we can find a subset $\R'$ of $\R$ consisting of at least $\frac{1}{2}(2m - 1)^{3d}$ words from $W'_{l}$ (and, in the same way as in the proof of Theorem \ref{th:theoremA}, it can be chosen to be uniformly random), our group $\Gamma'$ will then be a quotient of a group $\pres{\Ss}{\R'}$, with $\R'$ containing only words from $W'_{l}$, and we have shown that such groups typically have property (T), which finishes the proof.
\end{proof}

Actually we have covered only the case of $l$ divisible by $3$. This can be dealt with by relaxing the requirement that we choose relators of length exactly $l$ in the definition of the density model and instead allow all relators of length $\leq l$, which is arguably more natural (see discussion in \cite[I.2.c]{ollivier2005}). With high probability we can then find a sufficiently large subset of relators of length divisible by $3$ and repeat the proof above. Alternatively one could try to modify the triangular model and the homomorphism.

The final remark is that most of technical issues in proofs in this and the preceding section stem from the requirement that we choose only reduced words in the random presentations. If we allow all words, not necessarily reduced, to be chosen in the Gromov model, the straightforward construction of the homomorphism described in the beginning of this section works well. Similarly, if we allow nonreduced words in the triangular model, we do not have to deal with reductions and require that permutations are derangements when passing from the permutation model to the triangular model. For the sake of continuity with previous work in random groups we have adopted the traditional definitions of both models, but perhaps allowing all words to appear in the presentations would be more natural in many contexts.

\section{The triangular model - alternative approach}\label{ch4}
Here, we provide another proof that for $d > \frac{1}{3}$ random groups in the triangular model $\triangular{m}{d}$ have property (T) with overwhelming probability. The proof is based solely on spectral properties of random graphs, in particular almost regular graphs. In Section \ref{spectral}, we state basic definitions and lemmas about spectra of graphs. Then, in Section \ref{gnm-gnp}, we use results about the second smallest eigenvalue the of Laplacian $\lambda_1$ of a $G(n, p)$ random graph to deduce that a random graph in $G(n, M)$ has $\lambda_1 > \frac{1}{2}$ with high probability. In Section \ref{spectral-final}, we finish the proof by showing how to apply results about $G(n, M)$ model to the triangular model.
\subsection{Spectral properties of graphs}\label{spectral}

Here we state basic definitions and facts about spectra of random graphs that we will use later. A survey of basic facts from spectral graph theory can be found in \cite{chung}. Throughout this section we will work with $G(n, p)$ and $G(n, M)$ models described in Section \ref{ch3}.

For a graph $G = (V,E)$ with vertex set $V=\{v_1, \dots, v_n\}$, let $d(v_i)$ denote the degree of vertex $v_i$. The {\bf degree matrix} $D(G)$, or simply $D$, is the diagonal matrix with degrees of vertices on its diagonal, i.e. $D_{ii} = d(v_i)$ for $i=1, \dots, n$. We will assume that $G$ is connected and has no vertices of degree $0$. The {\bf adjacency matrix} of $G$, denoted by $A(G)$ or $A$, is the matrix such that its entry $A_{ij}$ is equal to the number of edges between $v_i$ and $v_j$. Since our graphs are unoriented, $A$ is symmetric, $A_{ij} = A_{ji}$.

\begin{definition}
The {\bf normalized Laplacian} $\Ll(G)$ is the matrix defined by
\[\Ll = I - D^{-\frac{1}{2}}AD^{-\frac{1}{2}}
\]
\end{definition}

Both $A$ and $\Ll$ are symmetric. For a symmetric matrix $M$, we can list its eigenvalues in increasing order: $\lambda_0(M) \leq \lambda_1(M) \leq \dots \leq \lambda_{n-2} \leq \lambda_{n - 1}(M)$. It can be easily seen that the spectrum  of $A(G)$ lies in the interval $[-d_{max}, d_{max}]$, where $d_{max} = \max_{v \in G}d(v)$ and the spectrum of $\Ll(G)$ lies in the interval $[0, 2]$. The first nonzero eigenvalue of $\Ll(G)$, $\lambda_1$, is often the called the {\it spectral gap} of $G$.

We can compare the notion of normalized Laplacian to the Laplacian $\Delta$ defined in Section \ref{ch2}, which can be expressed as $\Delta = I - D^{-1}A$. In the case of $d$-regular graphs, where all vertices have degree $d$, we have $D=dI$ and the matrices $\Ll, \Delta$ and $I-\frac{1}{d}A$ coincide. If we have a sequence of graphs which are ``almost regular'', in the sense that their minimal and maximal degrees are close to each other, we expect that spectral properties of $\Ll$ and $\Delta$ should be asymptotically the same, which will be made precise below. 

\begin{definition}
Let $\{G_{n}\}_{n=1}^{\infty}$ be a sequence of graphs and let $d_{n}$ be any sequence. We say that graphs $\{G_{n}\}_{n=1}^{\infty}$ are {\bf almost $d_{n}$-regular} if for every $G_{n}$ its minimum and maximum degree are $(1 + o(1))d_{n}$.
\end{definition}

We will often apply the term ``almost $d_{n}$-regular'' when talking about a single graph, where it is implicit that the graph depends on some parameter $n \to \infty$. In particular we will say that a random graph in $G(n, p)$ (or $G(n, M)$) is {\it almost $d_{n}$-regular with high probability} if its minimum and maximum degree is $(1 + o(1))d_{n}$ with high probability.

We will be interested in studying random graphs in $G(n, p)$, for which the average degree is $(n-1)p$. Below, this will be asymptotically the same as $np$ and for notational convenience we will replace $(n-1)p$ by $np$. For $p$ such that $\binom{n}{2}p = n^{1 + \varepsilon}$, a typical graph in $G(n, p)$ will be almost $np$-regular, thanks to the following easily proved fact \cite[Theorem 8.5.1]{alon}:

\begin{lemma}\label{lm:almost-regular}
Let $p = \omega(\frac{\log n}{n})$ and $G$ be a random graph in $G(n, p)$. Then with high probability for all $v \in G$ we have $d(v) = (1 + o(1))np$, i.e. $G$ is almost $np$-regular.
\end{lemma}

It can be easily shown that the largest eigenvalue of $A$ for an almost $d_{n}$-regular graph is $(1 + o(1))d_{n}$. The next standard lemma states that having the second largest eigenvalue of $A$ asymptotically smaller than the largest is equivalent to having spectral gap uniformly close to $1$:

\begin{lemma}\label{lm:adjacency-gap}
Let $d_n \rightarrow \infty$ and let $G_n$ be almost $d_{n}$-regular. Then $\frac{1}{d_n}\lambda_{n-2}(A(G_n)) = (1 + o(1))(1 - \lambda_1(\Ll(G_n)))$. In particular if $\lambda_{n-2}(A(G_n)) = o(d_{n})$ then $\lambda_{1}(\Ll(G_{n})) = 1 - o(1)$.
\end{lemma}

\begin{proof}
We shall use the following variational characterization of eigenvalues of a symmetric matrix $M$, called the Courant-Fischer Theorem \cite{chung}:
\[
\lambda_k = \max_{V^{k}}  \min_{\substack{x \perp V^{k} \\ \scalar{x}{x} = 1}} \scalar{Mx}{x}
\]
where the maximum is taken over all $k$-dimensional subspaces $V^{k}$.

It will be more convenient to replace $\Ll$ with $\Ll' = I - D^{-1}A$ - since these two matrices are similar, they share the same spectrum.

Note that since all vertex degrees are $(1+o(1))d_n$, all nonzero entries of $\frac{1}{d_n}D$ are $1 + o(1)$, so we can write:
\begin{align*}
\frac{\lambda_{n - 2}(A)}{d_n} &= \max_{V^{n-2}}  \min_{\substack{x \perp V^{n-2} \\ \scalar{x}{x} = 1}} \scalar{\frac{1}{d_n}Ax}{x} \\
&= \max_{V^{n-2}}  \min_{\substack{x \perp V^{n-2} \\ \scalar{x}{x} = 1}} \scalar{\frac{1}{d_n}DD^{-1}Ax}{x} \\
&\leq (1+ o(1)) \max_{V^{n-2}}  \min_{\substack{x \perp V^{n-2} \\ \scalar{x}{x} = 1}} \scalar{D^{-1}Ax}{x} \\
&= (1+ o(1)) \lambda_{n-2}(I-\Ll') = (1 + o(1)) (1 - \lambda_1(\Ll')) = (1 + o(1)) (1 - \lambda_1(\Ll))
\end{align*}
\qedhere
\end{proof}

The next lemma says that the Laplacian of an almost regular graph does not change much if we add a small ``correction'' to the graph. 

\begin{lemma}\label{lemma:haupt}
Let $G$ be an almost $d_{n}$-regular graph and let $K$ be a graph on the same vertex set whose maximum degree is $o(d_{n})$. Then: 
\begin{enumerate}[a)]
\item $G \cup K$ is almost $d_{n}$-regular  
\item $\lambda_{1}(\Ll(G)) = 1 - o(1)$ if and only if $\lambda_{1}(\Ll(G \cup K)) = 1 - o(1)$.
\end{enumerate}
\end{lemma}

\begin{proof}
Since $K$ has maximum degree $o(d_{n})$, every vertex of the graph $G \cup K$ has degree at most $(1 + o(1))d_{n} + o(d_{n}) = (1 + o(1))d_n$, so $G \cup K$ is almost $d_{n}$-regular, which proves a).

Now assume that $\lambda_{1}(\Ll(G)) = 1 - o(1)$. By Lemma \ref{lm:adjacency-gap}, we have $\lambda_{n-2}(A(G)) = o(d_{n})$. Since $K$ has maximum degree $o(d_{n})$, all eigenvalues of its adjacency matrix are $o(d_{n})$ in absolute value. We have $A(G \cup K) = A(G) + A(K)$, so for $i = 1, \ldots, n$ we obtain $\lambda_{i}(A(G \cup K)) = \lambda_{i}(A(G)) + o(d_{n})$ (since we are adding a matrix with spectral norm $o(d_{n})$). In particular $\lambda_{n-2}(G \cup K) = o(d_{n})$ so, as from a) $G \cup K$ is almost $d_{n}$-regular, by Lemma \ref{lm:adjacency-gap} we have $\lambda_{1}(\Ll(G \cup K)) = 1 - o(1)$.
\qedhere
\end{proof}

\subsection{From $G(n, p)$ to $G(n, M)$}\label{gnm-gnp}

In this section we show that if $M$ is sufficiently large, then typical graphs in $G(n, M)$ have good spectral properties, i. e. they have the spectral gap close to $1$ as $n \to \infty$.

It is well known that for $p$ and $M$ such that $p = \omega\left(\frac{\log n}{n}\right)$ and $M = \omega(n \log n)$ random graphs in $G(n, p)$ and $G(n, M)$ are connected with high probability \cite{bollobas2001random}. Since these bounds will be satisfied in all cases we are considering, we will assume that our graphs are connected.

We will rely on the following theorem about the $G(n, p)$ model (which is a corollary of a more general result from \cite[Theorem 3.6]{chung-gnp}):

\begin{theorem}\label{th:chung}
Suppose $np = \omega(\log^2 n)$ and let $g(n)$ be an arbitrary function tending to infinity. Then with high probability a random graph $G$ in $G(n, p)$ satisfies:
\[
\lambda_{1}(\Ll(G)) \geq 1 - \left(1 + o(1)\right)\frac{4}{\sqrt{np}} - \frac{g(n)\log^2 n}{np}
\]
\end{theorem} 

In particular this implies that with high probability $\lambda_{1}(\Ll(G)) = 1 - o(1)$. We would like to prove the analogous result for the $G(n, M)$ model.

First we will need a lemma concerning the maximum vertex degree in a sparse graph.

\begin{lemma}\label{lemma:gnm-edges}
Let $G$ be random graph in $G(n, k)$ with $k \leq \sqrt{M}$, where $M = \omega(n \log n)$. Then with high probability:
\[
\max\limits_{v \in G} d(v) = o\left(\frac{M}{n}\right)
\]
In particular all eigenvalues of $A(G)$ are with high probability $o\left(\frac{M}{n}\right)$.
\end{lemma}

\begin{proof}
Fix a vertex $v \in G$. We first estimate the probability that $v$ has degree at least $d$ when $G$ is from $G(n, p)$ for $k = p{n \choose 2}$. Note that for such $k$ we have $pn \leq \left(\frac{2\sqrt{M}}{n}\right)$, so the average degree $\bar{d}$ is at most $1$ (and actually tends to $0$ if $M = o(n^2)$). We use a tail estimate for the binomial distribution (see \cite[Theorem 2.1]{luczak}):
\begin{align}\label{binomial-tail}
\mathbb{P}\left(d(v) \geq \bar{d} + t\right) \leq \exp\left(-\frac{t^2}{2\left( \bar{d} + \frac{2}{3}t \right)}\right) 
\end{align}
The right hand side can be made $n^{-\omega(1)}$ by taking $t = \omega(\log n)$ such that at the same time $t = o\left(\frac{M}{n}\right)$. The probability that {\it any} vertex has degree at least $d$, denote this by $\mathbb{P}_{p}(d)$, is $n$ times the right hand side of (\ref{binomial-tail}), which is therefore also $n^{-\omega(1)}$.

Now take the analogous probability in $G(n, k)$, denote it by $\mathbb{P}_{k}(d)$. From a general inequality between probabilities for $G(n,k)$ and $G(n,p)$ models \cite[Theorem 2.2, iii)]{bollobas2001random}, we have $\mathbb{P}_{k}(d) \leq 3\sqrt{k}\mathbb{P}_{p}(d)$. Since $\mathbb{P}_{p}(d)$ is $n^{-\omega(1)}$ and $k \leq \sqrt{M} \leq n$, the right hand side goes to zero, so with high probability all vertices in $G$ from $G(n, k)$ have degree $o\left(\frac{M}{n}\right)$.
\end{proof}

\begin{theorem}\label{th:gnm-laplacian}
Suppose $M = \omega(n \log^2 n)$. Then with high probability a random graph $G$ in $G(n, M)$ satisfies:

\begin{enumerate}[a)]
\item $G$ is almost $np$-regular
\item $\lambda_{1}(\Ll(G)) = 1 - o(1)$
\end{enumerate}
\end{theorem}

\begin{proof}
Fix $M = \omega(n \log^2 n)$ and $p$ such that $M = p {n\choose 2}$. Let $H$ be a random graph in $G(n, p)$ conditioned on having between $M$ and $M + \sqrt{M}$ edges. Since we are conditiong on an event whose probability is bounded from below by some constant (which follows from the properties of binomial distribution), conclusion of Lemma \ref{lm:almost-regular}  and Theorem \ref{th:chung} still hold, so with high probability $H$ is almost $np$-regular and we have $\lambda_{1}(\Ll(H)) = 1 - o(1)$.

Now let $G$ be a subgraph of $H$ having exactly $M$ edges, chosen uniformly at random from all such subgraphs. $G$ has the same distribution as a random graph in $G(n, M)$, so it is enough to show that $\lambda_{1}(\Ll(G)) = 1 - o(1)$ with high probability. We will first estimate the eigenvalues of the adjacency matrix of $G$.

Let $K$ be a subgraph of $G$ consisting of edges not belonging to $H$, so that $A(K) = A(H) - A(G)$. If we condition on $H$ having exactly $M + k$ edges, for $0 \leq k \leq \sqrt{M}$, then we can think of this process as first choosing $K$ according to the $G(n, k)$ distribution (without any constraints) and then choosing $G$ uniformly among the graphs on the remaining  edges. As $K$ is a random graph in $G(n, k)$ with $k \leq \sqrt{M}$, from Lemma \ref{lemma:gnm-edges} we have that with high probability all vertices of $K$ have degree $o\left(\frac{M}{n}\right) = o\left(np\right)$.

Since the above property holds for all $k = 0, \ldots, \sqrt{M}$, by Lemma \ref{lemma:haupt}  $G$ is almost $np$-regular and $\lambda_{1}(\Ll(G)) = 1 - o(1)$ if and only if $\lambda_{1}(G \cup K) = \lambda_{1}(H) = 1 - o(1)$. As $\lambda_{1}(\Ll(H)) = 1 - o(1)$ with high probability, this finishes the proof.
\end{proof}

\subsection{Application to the triangular model}\label{spectral-final}

We will now proceed to prove Theorem \ref{th:theoremA} using results from the previous section. We use the notation of Section \ref{section:zuk-outline} --- $\lambda_{1}(G)$ denotes the smallest nonzero eigenvalue of $\Delta$ (equivalently: of $\Ll$) on graph $G$.

Fix $d > 1/3$ so that $3d = 1 + \varepsilon$. Let $\Gamma = \pres{\Ss}{\R}$ be a random group in the $\triangular{m}{d}$ model (we will asume that we actually choose $(2m)^{3d} \approx (2m-1)^{3d}$ relators in the triangular model) and let $L(\Ss)$ be the graph associated with $\Gamma$ as in Section \ref{section:zuk-outline}. Let us divide $L(\Ss)$ into three subgraphs $L_{1}$, $L_{2}, L_{3}$ as in Section \ref{section:zuk-outline}. 

Each $L_{i}$ is a random graph on $n = 2m$ vertices, with the marginal distribution the same as obtained by putting $n^{1 + \varepsilon}$ edges one by one uniformly at random. As such it is not a graph in $G(n, M)$ for $M = n^{1 + \varepsilon}$, as it allows duplicate edges. However, the number of duplicates is with high probability negligible as compared to $n^{1 + \varepsilon}$ --- choosing a random graph $L_{i}$ is equivalent to an experiment in which put $n^{1 + \varepsilon}$ balls into ${n\choose 2}$ bins at random and it can be easily calculated (see e.g. \cite[Chapter 5]{mitzenmacher}) that with high probability we will have no more than $O(n^{2\varepsilon})$ duplicate edges. Therefore if we denote by $L'_{i}$ a graph obtained from $L_{i}$ by collapsing all duplicate edges into one edge, the difference $K_{i} = L_{i}\backslash L'_{i}$ is a graph with $O(n^{2\varepsilon})$ edges. Its average degree is $o(1)$ and it can be easily shown, using the same methods as in proofs of Lemma \ref{lemma:haupt} and  Lemma \ref{lemma:gnm-edges}, that with high probability $K_{i}$ has maximum degree at most $o(n^{\varepsilon})$, so that $L_{i}$ and $L'_{i}$ will have asymptotically the same spectral properties. Therefore from now on we will treat $L_{i}$ as a random graph in $G(n, M)$, without duplicates.

\begin{lemma}\label{lm:laplacian-l1l2l3}
In the setting as above, with high probability $\lambda_{1}(L(\Ss)) > \frac{1}{2}$.
\end{lemma}

\begin{proof}
Let $\Delta_{i}$ denote the Laplacian on $L_{i}$, $\Delta$ --- the Laplacian of $L(\Ss)$, $D_{i}$ --- the degree matrix of $L_{i}$, $D$ --- the degree matrix of $L(\Ss)$ (in particular $D = D_{1} + D_{2} + D_{3}$). From the definition of $\Delta$ it is easily seen that:
\[
\Delta = \sum\limits_{i=1}^{3}D_{i}D^{-1}\Delta_{i}
\]
Since $M = n^{1 + \varepsilon} = \omega(n \log^2n)$, from part a) of Theorem \ref{th:gnm-laplacian} we have that each $L_{i}$ is with high probability almost $d_{n}$-regular for $d_{n} = n^{\varepsilon}$. In particular each nonzero entry of $D_{i}$ is $(1 + o(1))d_{n}$ and each nonzero entry of $D$ is $(3 + o(1))d_{n}$. Therefore we can write $D_{i}D^{-1} = \frac{1}{3}I + K_{i}$, where $K_{i}$ has only diagonal entries which are all $o(1)$. We get:
\[
\Delta = \frac{1}{3}\sum\limits_{i=1}^{3}\Delta_{i} + K
\]
where $K = \sum\limits_{i=1}^{3}K_{i}\Delta_{i}$ has all eigenvalues at most $o(1)$ (as $\Delta_{i}$ has spectrum in $[0, 2]$).

Now suppose that $\Delta$ has an eigenvector $f$ with eigenvalue $\leq 1/2$. We have:
\[
\frac{1}{2}\scalar{f}{f} \geq \scalar{\Delta f}{f} = \frac{1}{3}\sum\limits_{i=1}^{3}\scalar{\Delta_{i} f}{f} + \scalar{Kf}{f}
\]
By Theorem \ref{th:gnm-laplacian} $\lambda_{1}(L_{i})$ are uniformly bounded away from $1/2$ and $\scalar{Kf}{f}$ will be arbitrarily small as $n \to \infty$. Therefore the above inequality implies that for at least one $i$ we have $\scalar{\Delta_{i} f}{f} \leq \frac{1}{2}\scalar{f}{f}$, which contradicts $\lambda_{1}(L_{i}) > \frac{1}{2}$.
\end{proof}

\begin{proof}[Alternative proof of Theorem \ref{th:theoremA}]
By Theorem \ref{th:lambda-1/2} it is enough to show that $\lambda_{1}(L(\Ss)) > \frac{1}{2}$ with overwhelming probability.  Each $L_{i}$ is a random graph in $G(n, M)$ for $M \approx n^{1 + \varepsilon}$, in particular $M = \omega(n \log^2 n)$, so the assumptions of Theorem \ref{th:gnm-laplacian} are satisfied and for each $i$ we have $\lambda_{1}(L_{i}) = 1 - o(1)$ with high probability. By Lemma \ref{lm:laplacian-l1l2l3} we get $\lambda_{1}(L(\Ss)) > \frac{1}{2}$ with overwhelming probability and this finishes the proof.
\end{proof}

As we have seen in Section \ref{ch3}, to pass from the triangular model to the Gromov model it is necessary to use a modification of the triangular model, the {\it positive triangular model}, which results in $L(\Ss)$ being a random bipartite graph. Therefore what we actually need to pass to the Gromov model and prove Theorem \ref{th:theoremB} is to prove that groups in the positive model typically have property (T). To this end we would have to replace Theorem \ref{th:chung} with an analogous result for the $G(n, n, p)$ model of bipartite graphs (defined naturally). We have not found such a result in the literature, although it seems plausible that Theorem \ref{th:chung} also holds in the bipartite case.

\bibliography{bibliografia}{}
\bibliographystyle{amsalpha}

\end{document}